\documentclass[12pt]{iopart}

\usepackage{iopams}  
\usepackage{amsthm}

\def\a{{a}}

\def\g{r}

\def\y{g}
\def\x{f}

\def\P{T}

\def\e{e}
\def\R{\mathbb{R}}
\def\A{\mathcal{A}}

\newtheorem{thm}{Theorem}
\newtheorem{lem}[thm]{Lemma}

\begin{document}

\title{Non-convexly constrained linear inverse problems}

\author{Thomas Blumensath}

\address{Applied Mathematics,
School of Mathematics,
University of Southampton,
University Road,
Southampton, SO17 1BJ, UK}
\ead{thomas.blumensath@soton.ac.uk}
\begin{abstract}
This paper considers the inversion of ill-posed linear operators. To regularise the problem the solution is enforced to lie in a non-convex subset. Theoretical properties for the stable inversion are derived and an iterative algorithm akin to the projected Landweber algorithm is studied. This work extends recent progress made on the efficient inversion of finite dimensional linear systems under a sparsity constraint to the Hilbert space setting and to more general non-convex constraints. 
\end{abstract}

\maketitle

\section{Introduction}
\label{section:intro}

Let $\P$ be a linear map between two Hilbert spaces $H$ and $L$. For $\x\in H$ it is often required to recover $\x$ given $\P\x$. Furthermore, this inversion should be stable to small perturbations, that is, given
\begin{equation}
\label{equation:prob}
\y=\P\x+\e,
\end{equation} 
where $\y\in L$, $\e\in L$ and $\x \in H$ are elements in their respective Hilbert spaces, the estimate of $\x$ given $\y$ should linearly depend on the size of the error $\e$.
In many cases of interest, the operator $\P$ is ill-conditioned or non-invertible and some form of regularisation is required \cite{book00reg}. One possible approach to regualrisation is to assume that $\x$ lies in or close to a convex subset $\A\subset H$. Under certain conditions, stable inversion is then possible. We here extend this approach to more general non-convex sets $\A\subset H$. We define and study an 'optimal' solution to the inverse problem and propose a projected Landweber algorithm to efficiently calculate a solution with similar error bounds than those derived for the 'optimal' solution.

Our treatment is motivated by recent work in signal processing, where sparse signal models have gained increasing interest. It has been recognized that (under certain conditions) many finite dimensional inverse problems can be solved efficiently whenever the solution is sparse enough in some basis, that is, whenever the solution admits a representation in which most of its elements are zero or negligibly small \cite{donoho06compressed}, \cite{candes08RIP}, \cite{needell08COSAMP} and \cite{blumensath08IHT}. 

In sparse inverse problems, where $\x\in\R^N$ and $\y\in\R^M$, we assumed $\x$ to have only $K<<N$ non-zero elements. In the noiseless setting, that is if $\e=0$, it is then necessary to identify that point $\x$ in the affine subspace defined by $\P\x$, that has the fewest non-zero elements. The sparse inverse problem is a constrained inverse problem in which the constraint set is non-convex. More precisely, the constraint set is the union of ${N\choose K}$ $K$-dimensionl subspaces, where each subspaces is spanned by a different combination of $K$ canonical basis vectors. More general union of subspaces can be considered \cite{lu07theory} and \cite{blumensath09sampling}.

Solving sparse inverse problems is known to be NP-hard in general \cite{natarajan95sparse}. However, under certain conditions it could be shown that polynomial time algorithms can find near optimal solutions \cite{donoho06compressed}, \cite{candes08RIP}, \cite{needell08COSAMP} and \cite{blumensath08IHT}. We here extend these ideas to linear inverse problems in more general Hilbert spaces and study general non-convex constraint sets.

Given (\ref{equation:prob}) and assuming that $\P$ is ill-conditioned or non-invertible, we assume that $\x$ lies in a set $\A$, where $\A\subset H$. Importantly, we allow the set $\A$ to be non-convex in general.

In this general set-up, three important questions arise. 
\begin{enumerate}
\item Assuming $\x \in \A$ and $\e=0$, under which conditions can we recover $\x$ from $\P\x$?
\item Under which conditions is this inversion stable to small perturbations $\e$, that is, under which conditions can we guarantee that the inverse of $\P\x+\e$ is 'close' to the inverse of $\P\x$?
\item How and under which conditions can we calculate the inversion efficiently?
\end{enumerate}

The answer to the first question is that we require $\P$ to be one to one as a map from $\A$ to $L$. This is clearly necessary and sufficient for the recovery of $\x\in\A$ from $\P\x$. 

If $\A$ is a compact subset of Euclidean space, then we can lower bound the dimension of $L$ using Whitney's Embedding Theorem \cite[chapter 10]{introduction00lee} and its probabilistic version by Sauer et al. \cite{sauer91embedology} that states that
\begin{thm}[Theorem 2.3 \cite{sauer91embedology}]
Assume a compact subset $\A$ of $R^N$ with box counting dimension $k$ and let $M>2k$, then \emph{almost all} smooth maps $F$ from $\R^{N}$ to $\R^{M}$ have the property that $F$ is \emph{one to one} on $\A$. 
\end{thm}
Importantly, the above result also holds for almost all linear embeddings. In this case, the result can easily be generalized to non-compact subsets of $\R^N$ by considering the box counting dimension of the projection of the set $\A \setminus 0$ onto the unit sphere. 

However, the one to one property does not tell use anything about stability of the inverse, neither about how we would go about calculating such an inverse.

For $\x\in\A$, in order to guarantee stability to perturbations $\e$, it is necessary to impose a bi-Lipschitz condition on $\P$ as a map from $\A\subset H$ to $\P\A \subset L$, that is, we require that there exist constants $0<\alpha\leq\beta$, such that for all $\x_1, \x_2\in\A$ we have
\begin{equation}
\alpha \|\x_1+\x_2\|^2\leq  \|\P(\x_1+\x_2)\|^2\leq\beta \|\x_1+\x_2\|^2.
\end{equation}
The existence of such maps has been studied in \cite{lankarani05bi}.

Note that this condition guarantees that $\P$ is one to one as a function from $\A$ to $\P\A$. We are therefore able, at least in theory, to invert $\P$. The condition also guarantees stability in that if we are given an observations $\y=\P\x+\e=\P\x_\A+\P (\x-\x_\A)+\e$, where $\x_\A \in\A$, then we could, at least in theory, recover a good approximation of $\x$ by first choosing $\hat\y$ to be the 'projection' (see below how this can be defined for general subsets of Hilbert space) of $\y$ onto a close element in $\P\A$ and then taking $\hat\x\in\A$ so that $\hat\y=\P\hat\x$. We show below that the bi-Lipschitz property of $\P$ guarantees that $\hat\x$ is close to $\x$, where the distance depends on $\P(\x-\x_\A)+\e$. Note however that the bi-Lipschitz constant does not bound $\P(\x-\x_\A)$ in general, so that we also require $\P$ to be bounded, i.e. continuous.

Unfortunately, how the above recovery of $\x$ is to be done efficiently for general sets $\A$ is not clear. We here propose an algorithm and show that under certain conditions on $\alpha$ and $\beta$ we can efficiently calculate near optimal solutions. In order for our algorithm to be useful, we require that we are able to efficiently calculate the 'projection' of any $\x$ onto $\A$. 

In inverse problems, the operator $\P$ is generally given. However, we have to choose $\A$. On the one hand, we want to choose $\A$ in accordance with prior knowledge about reasonable solutions for a given problem. On the other hand, the stability considerations discussed above as well as the requirements in our theory developed below, show that we want the set $\A$ be chosen so that $\P$ is bi-Lipschitz as a map from $\A$ to $L$.

\section{An 'optimal' solution}

Let us here formalize the solution to the general inversion problem (\ref{equation:prob}), under the constraint that $\x\in\A$ as outlined above. We first define what we mean by a projection onto a general-non-convex set. For any $\y\in L$, consider
\begin{equation}
\inf_{\x\in\A}\|\y-\P\x\|^2.
\end{equation}
If $\A$ is non-empty, then by the properties of the infirmum and the fact that the norm is bounded from below, the above infirmum exists and is finite.
However, there might not exist an $\tilde\x\in\A$ such that $\|\y-\P\tilde\x\|^2=\inf_{\x\in\A}\|\y-\P\x\|^2$. In this case, we will consider estimates $\tilde\x$ that are close to this infirmum. 

The definition of the infirmum and the fact that $\inf_{\x\in\A}\|\y-\P\x\|^2<\infty$ prove the following trivial fact.
\begin{lem}
Let $\A$ be a nonempty closed subset of a Hilbert space $H$. Let $\P$ be a linear operator form $H$ into a Hilbert space $L$, then for all $\epsilon>0$ and $\y\in L$, there exist an element $\tilde\x\in\A$ for which 
\begin{equation}
\|\y-\P\tilde\x\|\leq\inf_{\x\in\A}\|\y-\P\x\| + \epsilon.
\end{equation}
\end{lem}

We can therefore define the set-valued mapping
\begin{equation}
m_{\A}^{\epsilon}(\y)=\{ \tilde\x : \|\y-\P\tilde\x\|^2\leq\inf_{\x\in\A}\|\y-\P\x\|^2 + \epsilon\}.
\end{equation} 
By the above lemma, the sets $m_{\A}^{\epsilon}(\x)$ are non-empty for all $\epsilon>0$, so that we can define $\epsilon$-optimal estimates as those points
\begin{equation}
\x_{opt}^{\epsilon} \in m_{\A}^{\epsilon}(\x).
\end{equation}

Similarly, we can define the set-valued mapping
\begin{equation}
p_{\A}^{\epsilon}(\x)=\{ \tilde\x : \|\x-\tilde\x\|^2\leq\inf_{\hat\x\in\A}\|\x-\hat\x\|^2 + \epsilon\}.
\end{equation} 
and the $\epsilon$-projection
\begin{equation}
P_{\A}^{\epsilon}(\x)= S (p_{\A}^{\epsilon}(\x)),
\end{equation} 
where $S$ is a selection operator that returns a single element from a set. The form of this operator is of no consequence for the rest of the discussion as long as it returns a single element.

We can now define $\hat\e = \y-\P\x_{opt}^\epsilon$, $\x_{\A}^\delta=P_{\A}^\delta(\x)$ and 
$\y-\P\x_\A^\delta = \tilde\e$. We have as a direct consequence of the definition of $\x_{opt}^\epsilon$ that
\begin{equation}
\label{equation:OptErrBound}
\|\y-\P\x_{opt}^\epsilon\| = \|\hat\e\|\leq \|\y-\P\x_\A^\delta\|+\sqrt{\epsilon}=\|\tilde\e\|+ \sqrt{\epsilon}.
\end{equation}

To bound the error of estimates $\x_{opt}^\epsilon$ we use the following lemma
\begin{lem}
\label{lem:xopt}
If $\P$ is a bi-Lipschitz map from a nonempty set $\A$ to $L$ with bi-Lipschitz constants $\alpha,\beta>0$, then
\begin{equation}
\|\x_{\A}^\delta-\x_{opt}^\epsilon\| \leq \frac{1}{\sqrt{\alpha}} \left[2 \|\tilde\e\| + \sqrt{\epsilon} \right]
\end{equation}
\end{lem}
\begin{proof}
\begin{eqnarray}
\|\x_{\A}^\delta-\x_{opt}^\epsilon\| 
&\leq & \frac{1}{\sqrt{\alpha}} \|\P\x_{\A}^\delta-\P\x_{opt}^\epsilon\| \nonumber \\
&\leq & \frac{1}{\sqrt{\alpha}} \left[ \|\y-\P\x_{opt}^\epsilon\| + \|\y-\P\x_{\A}^\delta\| \right]\nonumber \\
\end{eqnarray}
and the Theorem follows form (\ref{equation:OptErrBound}).
\end{proof}

This result implies 
\begin{thm}
If $\P$ is a bounded linear operator from $H$ to $L$ that satisfies the bi-Lipschitz condition as a map from $\A\subset H$ to $L$ with constants $\alpha,\beta>0$, then, for any $\delta,\epsilon>0$ the estimate $\x_{opt}^\epsilon$ satisfies the bound
\begin{equation}
\|\x-\x_{opt}^\epsilon\| \leq \frac{2}{\sqrt{\alpha}}  \|\P(\x-\x_\A^\delta)+\e\| + \inf_{\tilde\x\in\A}\|\x-\tilde\x\| + \frac{\sqrt{\epsilon}}{\sqrt{\alpha}}  + \sqrt{\delta}.
\end{equation}
\end{thm}
\begin{proof}
Using \ref{lem:xopt} and the triangle inequality we have
\begin{eqnarray}
\|\x-\x_{opt}^\epsilon\| &\leq&  \frac{1}{\sqrt{\alpha}} \left[2 \|\tilde\e\| + \sqrt{\epsilon} \right] +\|\x-\x_\A^\delta\|\nonumber \\
&\leq& \frac{2}{\sqrt{\alpha}}  \|\tilde\e\| + \inf_{\tilde\x\in\A}\|\x-\tilde\x\| + \frac{\sqrt{\epsilon}}{\sqrt{\alpha}}  + \sqrt{\delta}
\end{eqnarray}
from which the theorem follows by the definition of $\tilde\e$.
\end{proof}

\section{The Iterative Projection Algorithm}
\label{section:IPA}
In the previous section we have studied the estimate 
\begin{equation}
\x_{opt}^{\epsilon}= S(m_{\A}^{\epsilon}(\y)),
\end{equation}
and bounded the error $\|\x-\x_{opt}\|$ if $\P$ is bi-Lipschitz from $\A$ to $L$.
However, it is not clear how to calculate $\x_{opt}$ for general sets $\A$.
We therefore propose an iterative algorithm in which the above optimization is replaced by a series of $\epsilon$-projections.
Importantly, we show that, under certain conditions on the bi-Lipschitz constants $\alpha$ and $\beta$, this algorithm has a similar error bound to that of $\x_{opt}$ given above.

The Iterative Projection Algorithm is a generalization of the Projected Landweber Iteration \cite{eicke92iteration} to non-convex sets and an extension of the Iterative Hard Thresholding algorithm of \cite{kingsbury03iterative}, \cite{blumensath08thresh} and \cite{blumensath08IHT} to more general constrained inverse problems.

Given $\y$ and $\P$, let $\x^{0}=0$ and $\epsilon\geq\epsilon^n>0$. The Iterative Projection Algorithm is the iterative procedure defined by the recursion
\begin{equation}
\x^{n+1} = P_{\A}^{\epsilon^n}(\x^{n} + \mu \P^*(\y-\P\x^{n} )),
\end{equation}
where $\P^*$ is the adjoint of $\P$.

In many problems, calculation of $P_{\A}^\epsilon(\a)$ is much easier than a brute force search for $\x_{opt}^\epsilon$. For example, in the $K$-sparse model, $P_{\A}^\epsilon(\a)$ simply keeps the largest (in magnitude) $K$ elements of a sequence $\a$ and sets the other elements to zero. The next result shows that under certain conditions, not only does the algorithm calculate solutions with an error guarantee similar to that of $\x_{opt}$, it does so in a fixed number of iterations (depending only on a form of signal to noise ratio).

We have the following main result.
\begin{thm}
\label{thm:main}
Given $\y=\P\x+\e$ where $\x\in H$. Let $\A\subset H$ be non-empty such that $\P$ is bi-Lipschitz as a map from $\A$ to $L$ with constants $\alpha$ and $\beta$ that satisfy $\beta\leq \frac{1}{\mu}<1.5\alpha$, then, after
\begin{equation}
n^\star = \left\lceil 2 \frac{\ln(\delta\frac{\|\tilde\e\|+\sqrt{\frac{\epsilon}{2\mu}}}{\|\x_\A\|})}{\ln(2/(\mu\alpha)-2)} \right\rceil
\end{equation}
iterations, the Iterative Projection Algorithm calculates a solution $\x^{n^\star}$ satisfying
\begin{equation}
\|\x-\x^{n^\star}\| \leq (c^{0.5}+\delta)(\|\tilde\e\|+\sqrt{\frac{\epsilon}{2\mu}}) + \|\x_{\A}-\x\|.
\end{equation}
where $c \leq \frac{4}{3\alpha-2\mu}$ and $\tilde \e = \P(\x-\x_{\A})+e$.
\end{thm}

Note that this is of the same order as the bound for $\x_{opt}^\mu$.

The above theorem has been proved for the $K$-sparse model in \cite{blumensath08IHT} and for constraint sparse models in \cite{baraniuk08model}. Our main contribution is to show that it holds for general constrained inverse problems, as long as the bi-Lipschitz property holds with appropriate constants. To derive the result, we pursue a slightly different approach to that in \cite{blumensath08IHT} and \cite{baraniuk08model} and instead follow ideas of \cite{garg09grad}. The approach used here is basically that used for union of subspace models in \cite{blumensath09subspaces}.


We first establish the following lemma.
\begin{lem}
\label{lem:ProjOpt}
Let $\g=2\P^*(\y-\P\x^n)$ and $\x^{n+1}=P_{\A}^{\epsilon^n}(\x^n+\mu\P^*(\y-\P\x^n))$.  If $\frac{1}{\mu}\geq\beta$ then
\begin{eqnarray}
& &\|\y-\P\x^{n+1}\|^2 - \|\y-\P\x^n\|^2 \nonumber \\
& \leq& -Re{\langle(\x_\A-\x^n), \g\rangle} + \frac{1}{\mu}\|\x_\A-\x^n \|^2 +\frac{\epsilon^n}{\mu}
\end{eqnarray}
\end{lem}
\begin{proof}
We have
\begin{eqnarray}
& & -Re{\langle(\x-\x^n),\g\rangle} + \frac{1}{\mu}\|(\x-\x^n) \|^2 \nonumber \\
&=& \frac{1}{\mu}[ \|\x - \x^n - \frac{\mu}{2} \g\|^2 - (\mu/2)^2\|\g\|^2],
\end{eqnarray}
so that on the one hand
\begin{eqnarray}
& 		&\|\y-\P\x^{n+1}\|^2 - \|\y-\P\x^n\|^2 \nonumber \\
&  =  & -Re{\langle(\x^{n+1}-\x^n), \g\rangle} + \|\P(\x^{n+1}-\x^n) \|^2 \nonumber \\
&\leq & -Re{\langle(\x^{n+1}-\x^n), \g\rangle} + \frac{1}{\mu}\|(\x^{n+1}-\x^n) \|^2 \nonumber
\end{eqnarray}
and on the other hand
\begin{eqnarray}
& & \inf_{\x\in\A} -Re{\langle(\x-\x^n),\g\rangle} + \frac{1}{\mu}\|(\x-\x^n) \|^2 \nonumber \\
&=& \frac{1}{\mu}[ \inf_{\x\in\A} \|\x - \x^n - \frac{\mu}{2} \g\|^2 - (\mu/2)^2\|\g\|^2] \nonumber \\
&\geq& \frac{1}{\mu}[ \|\x^{n+1} - \x^n - \frac{\mu}{2} \g\|^2  - (\mu/2)^2\|\g\|^2 - \epsilon] \nonumber \\
&=& -Re{\langle(\x^{n+1}-\x^n),\g\rangle} + \frac{1}{\mu}\|(\x^{n+1}-\x^n) \|^2 - \frac{\epsilon}{\mu},
\end{eqnarray}
where the inequality comes from the definition of $P_{\A}^{\epsilon^n}$

We have thus shown that $\x^{n+1}=P^\epsilon_\A(\x^n+\frac{\mu}{2}\g)$ implies  
\begin{eqnarray}
-Re{\langle(\x^{n+1}-\x^n), \g\rangle} + \frac{1}{\mu}\|(\tilde\x-\x^n) \|^2  \nonumber \\ \leq -Re{\langle(\tilde\x-\x^n), \g\rangle} + \frac{1}{\mu}\|(\tilde\x-\x^n) \|^2 + \frac{\epsilon}{\mu}
\end{eqnarray}
 for all $\tilde\x\in\A$. Because $\x_\A\in\A$, this implies that
\begin{eqnarray}
-Re{\langle(\x_\A-\x^n),\g\rangle} + \frac{1}{\mu}\|(\x_\A-\x^n) \|^2 +\frac{\epsilon}{\mu} \nonumber \\ \geq -Re{\langle(\x^{n+1}-\x^n),\g\rangle} + \frac{1}{\mu}\|(\x^{n+1}-\x^n) \|^2.
\end{eqnarray} 
\end{proof}

\begin{proof}[Proof of Theorem \ref{thm:main}]
Using $\x_{\A}=P_{\A}^\epsilon(\x)$
\begin{equation}
\|\x-\x^{n+1}\|\leq \|\x_{\A}-\x^{n+1}\| + \|\x_{\A}-\x\|.
\end{equation}
where the bi-Lipschitz property implies that
\begin{eqnarray}
\|\x_{\A}-\x^{n+1}\|^2 
&\leq& \frac{1}{\alpha}  \|\P(\x_{\A}-\x^{n+1})\|^2. 
\end{eqnarray}
Furthermore
\begin{eqnarray}
& &\|\P(\x_\A-\x^{n+1}) \|^2 = \|\y -\P\x^{n+1} -\tilde\e \|^2 \nonumber \\ 
& = &\|\y-\P\x^{n+1}\|^2 + \|\tilde\e\|^2 -2Re{\langle\tilde\e,(\y-\P\x^{n+1})\rangle} \nonumber \\ 
&\leq& \|\y-\P\x^{n+1}\|^2 +\|\tilde\e\|^2 + \|\tilde\e\|^2 +\|\y-\P\x^{n+1}\|^2 \nonumber \\
& = & 2 \|\y-\P\x^{n+1}\|^2 +  2 \|\tilde\e\|^2,
\end{eqnarray}
where the last inequality follows from 
\begin{eqnarray}
& &-2Re{\langle\tilde\e,(\y-\P\x^{n+1})\rangle} \nonumber\\ 
&=& -\|\tilde\e + (\y-\P\x^{n+1})\|^2 +  \|\tilde\e\|^2+\|(\y-\P\x^{n+1})\|^2  \nonumber\\ 
&\leq& \|\tilde\e\|^2+\|(\y-\P\x^{n+1})\|^2.
\end{eqnarray} 

We will now show that under the Lipschitz assumption of the theorem, 
\begin{equation}
\|\y-\P\x^{n+1}\|^2\leq (\frac{1}{\mu}-\alpha)\|(\x_\A-\x^n)\|^2 + \|\tilde\e\|^2.
\end{equation}

We have
\begin{eqnarray}
& &  \| \y-\P\x^{n+1}\|^2 - \|\y-\P\x^n\|^2 \nonumber \\
&\leq & -Re{\langle(\x_\A-\x^n),\g\rangle} + \frac{1}{\mu}\|\x_\A-\x^n \|^2 +\frac{\epsilon}{\mu}\nonumber \\
& =   & -Re{\langle(\x_\A-\x^n),\g\rangle}  + \alpha \|\x_\A-\x^n \|^2 +(\frac{1}{\mu}-\alpha)\|\x_\A-\x^n \|^2 +\frac{\epsilon}{\mu}\nonumber \\ 
& \leq   & -Re{\langle(\x_\A-\x^n),\g\rangle} + \|\P(\x_\A-\x^n) \|^2 +(\frac{1}{\mu}-\alpha)\|\x_\A-\x^n \|^2 +\frac{\epsilon}{\mu}\nonumber \\ 
&   =  & \|\y-\P\x_\A\|^2-\|\y-\P\x^{n}\|^2 +(\frac{1}{\mu}-\alpha)\|\x_\A-\x^n \|^2 +\frac{\epsilon}{\mu} \nonumber \\ 
& = & \|\tilde\e\|^2-\|\y-\P\x^{n}\|^2 + (\frac{1}{\mu}-\alpha)\|(\x_\A-\x^n) \|^2 +\frac{\epsilon}{\mu}
\end{eqnarray}
where the first inequality is due to Lemma \ref{lem:ProjOpt}.

Combining the above inequalities shows that
\begin{equation}
\label{equation:errb}
\|\x_{\A}-\x^{n+1}\|^2 
\leq 2\left( \frac{1}{\mu\alpha}-1 \right)\|(\x_\A-\x^n) \|^2 + \frac{4}{\alpha}\|\tilde\e\|^2 +\frac{2\epsilon}{\mu\alpha},
\end{equation}
so that $2 (\frac{1}{\mu\alpha}-1)<1$ implies that
\begin{equation}
\|\x_{\A}-\x^{k}\|^2 
\leq \left(2\left(\frac{1}{\mu\alpha}-1\right)\right)^k\|\x_\A \|^2 + c\|\tilde\e\|^2 +c \frac{\epsilon}{2\mu}, 
\end{equation}
where $c \leq \frac{4}{3\alpha-2\frac{1}{\mu}}$.

The theorem then follows by the bound
\begin{eqnarray}
\|\x-\x^{k}\| 
&\leq& \sqrt{\left(2\frac{1}{\mu\alpha}-2\right)^k\|\x_\A \|^2 + c\|\tilde\e\|^2 +c \frac{\epsilon}{2\mu}} + \|\x_{\A}-\x\| \nonumber \\
&\leq& \left(2\frac{1}{\mu\alpha}-2\right)^{k/2} \|\x_\A \| + c^{0.5}\|\tilde\e\| + c^{0.5} \sqrt{\frac{\epsilon}{2\mu}} + \|\x_{\A}-\x\|, 
\end{eqnarray}
which means that after $n^\star = \left\lceil 2 \frac{\ln(\delta\frac{\|\tilde\e\|+\sqrt{\frac{\epsilon}{2\mu}}}{\|\x_\A\|})}{\ln(2/(\mu\alpha)-2)} \right\rceil$ iterations we have
\begin{equation}
\|\x-\x^{n^\star}\| \leq (c^{0.5}+\delta)(\|\tilde\e\|+\sqrt{\frac{\epsilon}{2\mu}}) + \|\x_{\A}-\x\|.
\end{equation}

\end{proof}


\section{Convergence}
Whilst we cannot yet show that the algorithm converges, we can show that it will 'converge' to a neighborhood of $\x_{opt}^\delta$.
\begin{thm}
Under the assumptions of Theorem \ref{thm:main} and using $\epsilon_n$-projections in iteration $n$, where $\epsilon_n\rightarrow 0$, then as $n\rightarrow \infty$
\begin{equation}
\|\x_{opt}^\delta-\x^n\|^2\rightarrow c\|\P(\x-\x_{opt}^\delta)+\e\|^2
\end{equation}
\end{thm}
\begin{proof}
Mirroring the derivation that let to (\ref{equation:errb}) but replacing $\x_\A$ by $\x_{opt}^\delta$ shows that
\begin{equation}
\|\x_{opt}^\delta-\x^{n+1}\|^2 
\leq 2\left( \frac{1}{\mu\alpha}-1 \right)\|(\x_{opt}^\delta-\x^n) \|^2 + \frac{4}{\alpha}\|\P(\x-\x_{opt}^\delta)+\e\|^2 +\frac{2\epsilon_n}{\mu\alpha}.
\end{equation}
 
If $2 (\frac{1}{\mu\alpha}-1)<1$, for each $N$ and $k>N$, 
\begin{eqnarray}
\label{thm:boundopt}
& &\|\x_{opt}^\delta-\x^{k}\|^2  \nonumber\\
&\leq& \left(2\left(\frac{1}{\mu\alpha}-1\right)\right)^{k-N}\|\x_{opt}^\delta-\x^N \|^2 + c\|\P(\x-\x_{opt}^\delta)+\e\|^2 + c \frac{\epsilon_N}{2\mu}, 
\end{eqnarray}
where again $c \leq \frac{4}{3\alpha-2\frac{1}{\mu}}$.
Because $\|\x_{opt}^\delta-\x^N \|^2$ is bounded from above by (\ref{thm:boundopt}), so that in the limit $k\rightarrow\infty$
\begin{equation}
\|\x_{opt}^\delta-\x^{k}\|^2 \rightarrow c\|\P(\x-\x_{opt}^\delta)+\e\|^2 + c \frac{\epsilon_N}{2\mu} 
\end{equation}
and this holds for all $N>0$, so that the result follows by letting $N\rightarrow\infty$.
\end{proof}

 -------------------------------------------------------------------------

\end{document}